\documentclass[10pt,reqno]{amsart}

\usepackage[dvipsnames]{xcolor}
\usepackage{amsmath,amssymb,amsthm,amsfonts,enumerate,tikz,bm}
\usepackage{mathrsfs}
\usetikzlibrary{matrix,arrows,positioning,calc}
\usepackage[all]{xy}
\usepackage[bookmarksnumbered,colorlinks]{hyperref}
\usepackage{url}
\numberwithin{equation}{section}
\usepackage{graphicx}
\usepackage[T1]{fontenc}
\usepackage{textcomp}
\usepackage{palatino,helvet}
\usepackage{footmisc}
\usepackage{rotating}
\usepackage{tikz-cd}
\usepackage{adjustbox}
\usepackage{scrextend}
\usepackage{soul}

\newtheorem{definition}{Definition}[section]
\newtheorem{remark}[definition]{Remark}

\newtheorem{theorem}[definition]{Theorem}
\newtheorem{proposition}[definition]{Proposition}
\newtheorem{lemma}[definition]{Lemma}
\newtheorem{corollary}[definition]{Corollary}

\theoremstyle{remark}

\usepackage{paralist}

\usepackage{scrextend}
\deffootnote{2em}{0em}{\thefootnotemark\quad}

\newcommand{\mcC}{\mathcal{C}}

\newcommand{\mbE}{\mathbb{E}} 
\newcommand{\mbF}{\mathbb{F}} 
\newcommand{\mfs}{\mathfrak{s}} 
\newcommand{\xdashrightarrow}[1]{\overset{#1}{\dashrightarrow}}

\newcommand{\Ab}{\mathrm{Ab}}
\newcommand{\op}{\mathrm{op}}

\newcommand{\Hom}{\mathrm{Hom}}

\makeatletter
\def\@seccntformat#1{%
  \protect\textup{\protect\@secnumfont
    \ifnum\pdfstrcmp{section}{#1}=0 \scshape\bfseries\fi
    \ifnum\pdfstrcmp{subsection}{#1}=0 \bfseries\fi
    \csname the#1\endcsname
    \protect\@secnumpunct
  }%
}
\makeatother

\begin{document}

\title{A characterization of closed subfunctors through $3\times 3$-lemma property in extriangulated categories}

\author{Juan C. Cala}
\address[Juan C. Cala]{Facultad de Ciencias, Universidad Nacional Aut\'onoma de M\'exico. Circuito Exterior, Ciudad Universitaria. CP04510. Mexico City, MEXICO. 
}
\email{jccalab@gmail.com}

\author{Shaira R. Hern\'andez}
\address[Shaira R. Hern\'andez]{Instituto de Matem\'aticas. Universidad Nacional Aut\'onoma de M\'exico. Circuito Exterior, Ciudad Universitaria. CP04510. Mexico City, MEXICO}
\email{shaira.hernandez@im.unam.mx}

\begin{abstract} Given an extriangulated category $(\mcC,\mbE,\mfs)$, we introduce the $3 \times 3$-lemma property for subfunctors of $\mbE$ and prove that an additive subfunctor $\mbF$ of $\mbE$ is closed if, and only if, it satisfies this condition. This characterization extends a well known result by A. Buan (for abelian categories) to extriangulated categories. As an application of this result, we get a new equivalent condition to describe saturated proper classes $\xi$ in $\mcC$.
\end{abstract}

\maketitle
\pagestyle{myheadings}
\markboth{\rightline {\scriptsize J. C. Cala and S. R. Hern\'andez}}
         {\leftline{\scriptsize A characterization of closed subfunctors through $3\times 3$-lemma property in extriangulated categories}}

\section*{\textbf{Introduction}}
In \cite{butler_classes_1961}, M. Butler and G. Horrocks introduced the notion of subfunctor of an additive bifunctor Ext in the context of abelian categories. Their main contribution was to establish a framework for performing relative homological algebra through subfunctors. More precisely, in order to achieve this, they defined the concept of \textit{closed subfunctor}. This approach is linked to ideas previously explored by D. Buchsbaum \cite{buchsbaum_note_1959, buchsbaum_satellites_1960}, where he demonstrated how one can perform homological algebra relative to a class of morphisms satisfying certain axioms.

On the other hand, following the definition of \textit{exact category} introduced by D. Quillen in \cite{quillen_higher_1973}, Dräxler et al. \cite[Proposition 1.4]{draxler_exact_1999} established the connection between closed subfunctors and exact structures. Building on this, A. Buan \cite[Lemma 2.2.1]{buan_closed_2001} characterized closed subfunctors using \textit{the $3 \times 3$-lemma property}, employing advanced techniques from the theory of exact categories, such as the Freyd–-Mitchell embedding theorem \cite[Appendix A]{keller_chain_1990}. It is worth noting that D. Buchsbaum \cite{buchsbaum_satellites_1960} and J. Fisher \cite{fisher_tensor_1968}, among others, were already aware of the usefulness of this property, but they were not familiar with this characterization. Recently, J. Cala \cite[Theorem 3.18]{cala_elementary_2024} proves this result only using elementary methods from the theory of abelian categories.

As a common generalization of abelian, exact and triangulated categories,  H. Nakaoka and Y. Palu introduced in \cite{NP19} the so-called \textit{extriangulated categories}. Moreover, similar as the exact categories (in the sense of Quillen) axiomatize extension-closed subcategories from an abelian category, they provide an axiomatic framework for extension-closed subcategories of triangulated categories. Therefore, in recent years, several notions and constructions have been generalized to this new context. 

Relative theories for a higher dimensional analogue to extriangulated categories, namely the $n$-exangulated categories, have been studied in \cite[Section 3]{HLN21} by  M. Hershend, Y. Lui and H. Nakaoka. Their results contains the extriangulated context by setting $n=1$ (see \cite[Proposition 4.3]{HLN21}). Hence, for an extriangulated category $(\mcC,\mbE,\mfs)$,  \cite[Proposition 3.16]{HLN21} provides several equivalent conditions for an additive subfunctor $\mbF$ of $\mbE$ to be closed. In particular, it shows that closed subfunctors are exactly those $\mbF \subseteq \mbE$ for which $(\mbF,\mfs|_{\mbF})$  gives an external triangulation of $\mcC$, where $\mfs|_{\mbF}$ is the restriction of $\mfs$ to $\mbF$-extensions. Therefore, closed subfunctors provides a general method to construct new extriangulated structures from a given one. This has motivated the study of them in several works, see for example \cite{baillargeon2022lattices, enomoto2021classifying, iyama2024auslander,sakai2023relative,tattar2024structure}.

On the other hand, \textit{proper classes of $\mbE$-triangles} introduced by J. Hu, D. Zhang and P. Zhou in \cite{HZZ20} have been used to study different existing notions in exact and triangulated contexts but now for extriangulated categories, for example \cite{HZZ20, hu2021gorenstein,xu2025ideal}. This generalizes the idea of consider certain classes of triangles to develop relative homological algebra in triangulated categories as A. Beligiannis did in \cite{beligiannis2000relative}.
From A. Sakai's work in \cite{sakai2023relative}, we know that there is a bijective correspondence between closed subfunctors of $\mbE$ and proper classes of $\mbE$-triangles. This means that both approaches for studying relative theories in extriangulated categories are the same.

A natural question is whether closed subfunctors of $\mathbb{E}$ can be characterized via the 3×3-lemma property, in analogy with the approach taken by A. Buan in the abelian setting. To the best of our knowledge, no version of the Freyd–-Mitchell embedding theorem is currently available for extriangulated categories, and the uniqueness of universal properties does not hold in this context. These limitations prevent a straightforward adaptation of the methods employed by A. Buan \cite{buan_closed_2001} and J. Cala \cite{cala_elementary_2024}. The aim of this work is to answer that question positively by providing a proof that uses only elementary properties of closed subfunctors in extriangulated categories.

\subsection*{Organization of the paper} 
 In Section \ref{sec:preliminaries}, we recall some notions and properties of extriangulated categories from \cite{NP19}. We include without proofs some of the results that are used later on in the proof of our main result. On the other hand, given an extriangulated category $(\mcC,\mbE,\mfs)$, we recall from \cite{HLN21} in Definition \ref{dfn:subfunctor} what it means for a functor $\mbF$ to be a subfunctor of $\mbE$ and in Proposition \ref{HLN21:prop 3.16} we recall the relation between closed subfunctors of $\mbE$ and relative extriangulated structures of $(\mcC,\mbE,\mfs)$.

In Section \ref{sec:main}, for a given extriangulated category $(\mcC,\mbE,\mfs)$, we introduce in Definition \ref{dfn: 3x3 property} the \textit{$3\times 3$-lemma property} for subfunctors $\mbF$ of $\mbE$ and state our main result, namely, that closed subfunctors $\mbF$ of $\mbE$ are exactly those having the $3\times 3$-lemma property. Finally, we see how this idea fits into the theory of proper classes developed in \cite{HZZ20,sakai2023relative}.


\subsection*{Conventions}
Throughout this paper, $\mcC$ denotes an additive category. We write $A\in \mcC$ to say that $A$ is an object of $\mcC$. Given $A,B\in\mcC$, we denote by $\Hom_{\mcC}(A,B)$ the abelian group of morphism from $A$ to $B$ in $\mcC$. Finally, $\mathrm{Set}$ and $\mathrm{Ab}$ stand for the categories of sets and abelian groups, respectively.


\section{\textbf{Preliminaries}}\label{sec:preliminaries}

\noindent \textbf{Extriangulated categories and terminology.} Let us briefly recall some important concepts and establish the notation we will be using through this work. For more details about these concepts and properties related to extriangulated categories, we refer the reader to \cite[Sections 2 and 3]{NP19}. 

An \textit{extriangulated category} consists of a triplet $(\mcC,\mbE,\mfs)$, where $\mcC$ is an additive category, $\mbE:\mcC^{\op} \times \mcC \to \Ab$ is an additive bifunctor and $\mfs$ is an \textit{additive realization} of $\mbE$ satisfying certain axioms (see \cite[Definition 2.12]{NP19} for more details). 
An \textit{$\mbE$-extension} is a triplet $(A,\delta,C)$, where $A,C \in \mcC$ and $\delta \in \mbE(C,A)$. We simply say that $\delta$ is an $\mbE$-extension to refer that $\delta \in \mbE(C,A)$ for some $A,C \in \mcC$. For every $a \in \Hom_{\mcC}(A,A')$ and  $c \in \Hom_{\mcC}(C',C)$, we use the notation $a \cdot \delta$ and $\delta \cdot c$ to denote $\mbE(C,a)(\delta)$ and $\mbE(c,A)(\delta)$, respectively. Given $\delta \in \mbE(C,A)$ and $\delta' \in \mbE(C',A')$, a \textit{morphism of $\mbE$-extensions} from $\delta$ to $\delta'$ is a pair of morphisms $a \in \Hom_{\mcC}(A,A')$ and $c \in \Hom_{\mcC}(C',C)$ such that $a\cdot \delta=\delta'\cdot c$. In this case, we write  $(a,c):\delta \to \delta'$. 

The \textit{realization} $\mfs$ of $\mbE$ assigns to any $\delta \in \mbE(C,A)$ an equivalence class $\mfs(\delta)=[A \xrightarrow{x} B \xrightarrow{y} C]$ of a sequence of morphisms in $\mcC$ (see \cite[Definitions 2.7 and 2.9]{NP19}). In this situation, we say that $A \xrightarrow{x} B \xrightarrow{y} C$ \textit{realizes} $\delta$ and we will call it an $\mfs$-\textit{conflation}. Furthermore, we say that the pair $(A \xrightarrow{x} B \xrightarrow{y} C, \delta)$ is an \textit{$\mbE$-triangle} which we denote by $A \xrightarrow{x} B \xrightarrow{y} C \xdashrightarrow{\delta}$. For the zero element $0=0_{(C,A)} \in \mbE(C,A)$, we say that $A \xrightarrow{x} B \xrightarrow{y} C \xdashrightarrow{0}$ is a \textit{split} $\mbE$-triangle. Hence, since $\mfs$ is additive, we have the split $\mbE$-triangles $0\to C \xrightarrow{1_C}C \xdashrightarrow{0}$ and $A \xrightarrow{1_A}A \to 0 \xdashrightarrow{0}$ for any $A,C \in \mcC$. Let $\mfs(\delta)=[A \xrightarrow{x} B \xrightarrow{y} C]$ and $\mfs(\delta')=[A' \xrightarrow{x'} B' \xrightarrow{y'} C']$.  We say that a triplet $(a,b,c)$ of morphisms in $\mcC$ \textit{realizes} $(a,c):\delta \to \delta'$ if  $bx=x'a$ and $cy=y'b$ hold. The triplet $(a,b,c)$ is called a \textit{morphism of $\mbE$-triangles} and it is denoted by the following commutative diagram 
\[\begin{tikzcd}
	A & B & C & {{}} \\
	{A'} & {B'} & {C'} & {{}}.
	\arrow["x", from=1-1, to=1-2]
	\arrow["a"', from=1-1, to=2-1]
	\arrow["y", from=1-2, to=1-3]
	\arrow["b", from=1-2, to=2-2]
	\arrow["\delta", dashed, from=1-3, to=1-4]
	\arrow["c", from=1-3, to=2-3]
	\arrow["{x'}"', from=2-1, to=2-2]
	\arrow["{y'}"', from=2-2, to=2-3]
	\arrow["{\delta'}"', dashed, from=2-3, to=2-4]
\end{tikzcd}\]

The following outcome will be useful in the proof of our main result.

\begin{proposition}\cite[Proposition 3.15]{NP19}\label{NP19:prop 3.15} Let $(\mcC,\mbE,\mfs)$ be an extriangulated category. Let $A_1 \xrightarrow{x_1} B_1 \xrightarrow{y_1} C \xdashrightarrow{\delta_1}$ and $A_2 \xrightarrow{x_2} B_2 \xrightarrow{y_2} C \xdashrightarrow{\delta_2}$ be any pair of $\mbE$-triangles. Then, there is a commutative diagram of $\mbE$-triangles 
\begin{equation}  \label{diag:NP3.15}
\begin{tikzcd}[ampersand replacement=\&]
	\& {A_2} \& {A_2} \\
	{A_1} \& M \& {B_2} \& {{}} \\
	{A_1} \& {B_1} \& C \& {{}} \\
	\& {{}} \& {{}}
	\arrow[equals, from=1-2, to=1-3]
	\arrow["{m_2}"', from=1-2, to=2-2]
	\arrow["{x_2}", from=1-3, to=2-3]
	\arrow["{m_1}", from=2-1, to=2-2]
	\arrow[equals, from=2-1, to=3-1]
	\arrow["{e_1}", from=2-2, to=2-3]
	\arrow["{e_2}"', from=2-2, to=3-2]
	\arrow["{\delta_1\cdot y_2}", dashed, from=2-3, to=2-4]
	\arrow["{y_2}", from=2-3, to=3-3]
	\arrow["{x_1}", from=3-1, to=3-2]
	\arrow["{y_1}", from=3-2, to=3-3]
	\arrow["{\delta_2\cdot y_1}"', dashed, from=3-2, to=4-2]
	\arrow["{\delta_1}", dashed, from=3-3, to=3-4]
	\arrow["{\delta_2}", dashed, from=3-3, to=4-3]
\end{tikzcd}\end{equation}
which satisfies $m_1\cdot \delta_1+m_2\cdot \delta_2=0.$
\end{proposition}


\noindent \textbf{Closed subfunctors.} In what follows, we recall the notion of  closed subfunctors of $\mbE$ and establish some notation related to them.

\begin{definition}\cite[Definition 3.7]{HLN21}\label{dfn:subfunctor} Let $(\mcC ,\mbE,\mfs)$ be an extriangulated category. A functor $\mbF:\mcC^{\op} \times \mcC \to \mathrm{Set}$ is called a \textbf{subfunctor of $\mbE$} if it satisfies the following conditions:
\begin{enumerate}
    \item $\mbF(C,A)  \subseteq \mbE(C,A)$, for all $A,C \in \mcC.$
    \item $\mbF(c^{\op},a) = \mbE(c^{\op},a)\mid_{\mbF(C,A)}$ holds, for any $c \in \Hom_{\mcC}(C',C)$ and $\Hom_{\mcC}(A,A')$.
\end{enumerate}
In this case, we write $\mbF \subseteq \mbE$. If in addition, $\mbF(C,A)$ is an abelian subgroup of $\mbE(C,A)$ for any $C,A \in \mcC$, then we say that $\mbF$ is an \textbf{additive subfunctor} of $\mbE$.
\end{definition}

Let $\mbF$ be an additive subfunctor of $\mbE$. For a realization $\mfs$ of $\mbE$, we define $\mfs|_{\mbF}$ as the restriction of $\mfs$ onto $\mbF$, that is, $\mfs|_{\mbF}(\delta):=\mfs(\delta)$ for any $\mbF$-extension $\delta$. Thus, we say that $A \xrightarrow{x} B \xrightarrow{y}C$ is an \textit{$\mfs|_{\mbF}$-conflation} and $A \xrightarrow{x} B \xrightarrow{y}C \xrightarrow{\delta}$ is an \textit{$\mbF$-triangle}, if $A \xrightarrow{x} B \xrightarrow{y}C$ realizes an $\mbF$-extension $\delta$. In this case,  $x\colon A \to B$ and $y\colon B \to C$  are called \textit{$\mfs|_{\mbF}$-inflation} and \textit{$\mfs|_{\mbF}$-deflation}, respectively. 

\begin{definition}\cite[Definition 3.10]{HLN21} Let $\mbF$ be an additive subfunctor of $\mbE$. 
\begin{enumerate}
    \item $\mbF$ is \textbf{closed on the right} if, for any $\mfs|_{\mbF}$-conflation $A \xrightarrow{x} B \xrightarrow{y} C$ and $X \in \mcC$, the sequence
    \[ \mbF(X,A)\overset{\mathbb{F}(X,x)}{\longrightarrow} \mbF (X,B)\overset{\mathbb{F}(X,y)}{\longrightarrow}\mathbb{F}(X,C) \]
    is exact in $\mathrm{Ab}$.
    \item $\mbF$ is \textbf{closed on the left} if, for any $\mfs|_{\mbF}$-conflation $A \xrightarrow{x} B \xrightarrow{y} C$ and $X \in \mcC$, the sequence
     \[ \mbF(C,X)\overset{\mathbb{F}(y,X)}{\longrightarrow} \mbF (B,X)\overset{\mathbb{F}(x,X)}{\longrightarrow}\mathbb{F}(A,X) \]
    is exact in $\mathrm{Ab}$.
\end{enumerate}
We say that $\mbF$ is \textbf{closed} if it is both closed on the right and on the left (see \cite[Lemma 3.15]{HLN21}).
\end{definition}

The following characterization of closed subfunctors is the main result of \cite[Section 3.2]{HLN21}. This shows the relevance of studying closed subfunctors of $\mbE$ since they induce relative extriangulated structures.

\begin{proposition}\cite[Proposition 3.16]{HLN21} \label{HLN21:prop 3.16}  Let $(\mcC,\mbE,\mfs)$ be an extriangulated category. For any additive subfunctor $\mbF$ of $\mbE$, the following conditions are equivalent.
\begin{enumerate}
    \item $\mbF$ is closed.
    \item The class of $\mfs|_{\mbF}$-inflations is closed under composition.
    \item The class of $\mfs|_{\mbF}$-deflations is closed under composition.
    \item $(\mcC,\mbF,\mfs|_{\mbF})$ is an extriangulated category.
\end{enumerate}
\end{proposition}

\section{\textbf{Main Result}}\label{sec:main}
We begin this section by introducing the so-called $3\times3$-lemma property for subfunctors of $\mbE$ in an extriangulated category $(\mcC,\mbE,\mfs)$ and we study its relationship with closedness condition. The following outcome plays an important role for the rest of this work, so we include its proof below.

\begin{lemma}\label{lemma: Finfl-Fdefl}
    Let $(\mcC,\mbE,\mfs)$ be an extriangulated category, $\mbF$ be an additive subfunctor of $\mbE$ and $A \xrightarrow{f} B \xrightarrow{g} C \xdashrightarrow{\delta}$ be an $\mbE$-triangle. Then, $f$ is an $\mfs|_{\mbF}$-inflation if, and only if, $g$ is an $\mfs|_{\mbF}$-deflation.
\end{lemma}

\begin{proof}
    If $f$ is an $\mfs|_{\mbF}$-inflation, then there exists an $\mbF$-triangle $A \xrightarrow{f} B \xrightarrow{g^{\prime}} C^{\prime} \xdashrightarrow{\eta}$. By (ET3), there exists a morphism $c\colon C\to C^{\prime}$ in $\mcC$ making the following diagram commutative
    \[\begin{tikzcd}
	A & B & C & {} \\
	A & B & {C^{\prime}} & {},
	\arrow["f", from=1-1, to=1-2]
	\arrow[equals, from=1-1, to=2-1]
	\arrow["g", from=1-2, to=1-3]
	\arrow[equals, from=1-2, to=2-2]
	\arrow["\delta", dashed, from=1-3, to=1-4]
	\arrow["c", dashed, from=1-3, to=2-3]
	\arrow["f"', from=2-1, to=2-2]
	\arrow["{g^{\prime}}"', from=2-2, to=2-3]
	\arrow["\eta"', dashed, from=2-3, to=2-4]
\end{tikzcd}\] and then $\eta \cdot c=\delta$. Since $\mbF$ is an additive subfunctor, we have that $\delta\in\mbF(A,C)$ and so $g$ is an $\mfs|_{\mbF}$-deflation. Dually, the converse implication also holds.
\end{proof}

The following definition establishes the analogue of $3 \times 3$-lemma property in extriangulated categories. This concept originally appeared in \cite{buan_closed_2001} for abelian categories.

\begin{definition}\label{dfn: 3x3 property}
    Let $(\mcC,\mbE,\mfs)$ be an extriangulated category and let $\mbF$ be an additive subfunctor of $\mbE$. We say that $\mbF$ has the \textbf{$3\times 3$-lemma property} if for every commutative diagram of $\mbE$-triangles 
    \[\begin{tikzcd}
	A & B & C & {} \\
	D & E & G & {} \\
	H & I & J & {} \\
	{} & {} & {}
	\arrow["a", from=1-1, to=1-2]
	\arrow["i"', from=1-1, to=2-1]
	\arrow["b", from=1-2, to=1-3]
	\arrow["j", from=1-2, to=2-2]
	\arrow["{\varepsilon_1}", dashed, from=1-3, to=1-4]
	\arrow["c", from=1-3, to=2-3]
	\arrow["d"', from=2-1, to=2-2]
	\arrow["k"', from=2-1, to=3-1]
	\arrow["e"', from=2-2, to=2-3]
	\arrow["l", from=2-2, to=3-2]
	\arrow["{\varepsilon_2}", dashed, from=2-3, to=2-4]
	\arrow["f", from=2-3, to=3-3]
	\arrow["g"', from=3-1, to=3-2]
	\arrow["{\eta_1}"', dashed, from=3-1, to=4-1]
	\arrow["h"', from=3-2, to=3-3]
	\arrow["{\eta_2}", dashed, from=3-2, to=4-2]
	\arrow["{\varepsilon_3}", dashed, from=3-3, to=3-4]
	\arrow["{\eta_3}", dashed, from=3-3, to=4-3]
\end{tikzcd}\] where $(i,j,c)$ is a morphism of $\mbE$-triangles, if all the columns and outer rows are $\mbF$-triangles, then the middle row is also an $\mbF$-triangle.
\end{definition}

As we mentioned in the previous section, studying closed subfunctors of $\mbE$ is relevant as they induce relative extriangulated structures. The $3\times 3$-lemma property allows us to add another equivalent condition to those given in Proposition \ref{HLN21:prop 3.16}, which generalizes \cite[Lemma 2.2.1]{buan_closed_2001} to the context of extriangulated categories.


\begin{theorem} \label{thm:3x3 equiv closed} Let $(\mcC,\mbE,\mfs)$ be an extriangulated category. For any additive subfunctor $\mbF$ of $\mbE$, the following conditions are equivalent.
    \begin{enumerate}
        \item $\mbF$ has the $3\times 3$-lemma property.
        \item $\mbF$ is closed. 
    \end{enumerate} 
\end{theorem}

\begin{proof}
    $(1)\Rightarrow(2)$.\ According to Proposition \ref{HLN21:prop 3.16}, it suffices to see that the class of $\mfs|_{\mbF}$-deflations is closed under composition. Let $A \xrightarrow{a} B \xrightarrow{b} C \xdashrightarrow{\delta}$ and $D \xrightarrow{d} E \xrightarrow{e} B\xdashrightarrow{\eta}$ be two $\mbF$-triangles. By (ET4)$^{\op}$, there is a commutative diagram of $\mbE$-triangles 
    \[\begin{tikzcd}
	D & D \\
	G & E & C & {} \\
	A & B & C & {} \\
	{} & {}
	\arrow[equals, from=1-1, to=1-2]
	\arrow["f"', from=1-1, to=2-1]
	\arrow["d", from=1-2, to=2-2]
	\arrow[from=2-1, to=2-2]
	\arrow["g"', from=2-1, to=3-1]
	\arrow[from=2-2, to=2-3]
	\arrow["e", from=2-2, to=3-2]
	\arrow["\theta", dashed, from=2-3, to=2-4]
	\arrow[equals, from=2-3, to=3-3]
	\arrow["a"', from=3-1, to=3-2]
	\arrow["\mu"', dashed, from=3-1, to=4-1]
	\arrow["b"', from=3-2, to=3-3]
	\arrow["\eta", dashed, from=3-2, to=4-2]
	\arrow["\delta"', dashed, from=3-3, to=3-4]
\end{tikzcd}\] such that $g\cdot\theta=\delta$ and $\eta\cdot a=\mu$. This diagram can be completed up to a commutative diagram of $\mbE$-triangles as follows: 
\begin{equation}\label{eq:diag1}
    \begin{tikzcd}
	D & D & 0 & {} \\
	G & E & C & {} \\
	A & B & C & {} \\
	{} & {} & {}
	\arrow[equals, from=1-1, to=1-2]
	\arrow["f"', from=1-1, to=2-1]
	\arrow[from=1-2, to=1-3]
	\arrow["d", from=1-2, to=2-2]
	\arrow["{0_{0,D}}", dashed, from=1-3, to=1-4]
	\arrow[from=1-3, to=2-3]
	\arrow[from=2-1, to=2-2]
	\arrow["g"', from=2-1, to=3-1]
	\arrow[from=2-2, to=2-3]
	\arrow["e", from=2-2, to=3-2]
	\arrow["\theta", dashed, from=2-3, to=2-4]
	\arrow[equals, from=2-3, to=3-3]
	\arrow["a"', from=3-1, to=3-2]
	\arrow["\mu"', dashed, from=3-1, to=4-1]
	\arrow["b"', from=3-2, to=3-3]
	\arrow["\eta", dashed, from=3-2, to=4-2]
	\arrow["\delta"', dashed, from=3-3, to=3-4]
	\arrow["{0_{C,0}}", dashed, from=3-3, to=4-3]
    \end{tikzcd}
\end{equation}
    Since $\mu= \eta \cdot a$ is an $\mbF$-extension, then all columns and the first and third rows in diagram \eqref{eq:diag1} are $\mbF$-triangles. Given that $\mbF$ has the $3\times3$-lemma property, it follows that $G \xrightarrow{} E \xrightarrow{be} C \xdashrightarrow{\theta}$ is an $\mbF$-triangle and so the morphism $be$ is an $\mfs|_{\mbF}$-deflation.\\
    $(2)\Rightarrow(1)$. Let $\mbF$ be a closed subfunctor of $\mbE$. Let us assume that we have a commutative diagram of $\mbE$-triangles 
    \begin{equation}\label{eq:3x3(*)}
        \begin{tikzcd}
	A & B & C & {} \\
	D & E & G & {} \\
	H & I & J & {} \\
	{} & {} & {}
	\arrow["a", from=1-1, to=1-2]
	\arrow["i"', from=1-1, to=2-1]
	\arrow["b", from=1-2, to=1-3]
	\arrow["j", from=1-2, to=2-2]
	\arrow["{\varepsilon_1}", dashed, from=1-3, to=1-4]
	\arrow["c", from=1-3, to=2-3]
	\arrow["d"', from=2-1, to=2-2]
	\arrow["k"', from=2-1, to=3-1]
	\arrow["e"', from=2-2, to=2-3]
	\arrow["l", from=2-2, to=3-2]
	\arrow["(\ast)"{anchor=center, rotate=0}, draw=none, from=2-2, to=3-3]
	\arrow["{\varepsilon_2}", dashed, from=2-3, to=2-4]
	\arrow["f", from=2-3, to=3-3]
	\arrow["g"', from=3-1, to=3-2]
	\arrow["{\eta_1}"', dashed, from=3-1, to=4-1]
	\arrow["h"', from=3-2, to=3-3]
	\arrow["{\eta_2}", dashed, from=3-2, to=4-2]
	\arrow["{\varepsilon_3}", dashed, from=3-3, to=3-4]
	\arrow["{\eta_3}", dashed, from=3-3, to=4-3]
    \end{tikzcd}
    \end{equation} where $\varepsilon_1,\varepsilon_3,\eta_1,\eta_2,\eta_3$ are all $\mbF$-extensions and such that $i\cdot \varepsilon_1=\varepsilon_2\cdot c$. In order to see that $D \xrightarrow{d} E \xrightarrow{e} G \xdashrightarrow{\varepsilon_2}$ is an $\mbF$-triangle, we shall show that $d$ is an $\mfs|_{\mbF}$-inflation.\\
By the dual of Proposition \ref{NP19:prop 3.15}, we have the following commutative diagram \[\begin{tikzcd}
	A & B & C & {} \\
	D & M & C & {} \\
	H & H && {} \\
	{} & {} & {}
	\arrow["a", from=1-1, to=1-2]
	\arrow["i"', from=1-1, to=2-1]
	\arrow["b", from=1-2, to=1-3]
	\arrow["q", from=1-2, to=2-2]
	\arrow["{\varepsilon_1}", dashed, from=1-3, to=1-4]
	\arrow[equals, from=1-3, to=2-3]
	\arrow["u"', from=2-1, to=2-2]
	\arrow["k"', from=2-1, to=3-1]
	\arrow["x"', from=2-2, to=2-3]
	\arrow["p", from=2-2, to=3-2]
	\arrow["{i\cdot\varepsilon_1}"', dashed, from=2-3, to=2-4]
	\arrow[equals, from=3-1, to=3-2]
	\arrow["{\eta_1}"', dashed, from=3-1, to=4-1]
	\arrow["{a\cdot\eta_1}", dashed, from=3-2, to=4-2]
\end{tikzcd}.\] 
Now $i\cdot\varepsilon_1\in\mbF(C,D)$ and $a \cdot\eta_1\in\mbF(H,B)$ since $\mbF$ is an additive subfunctor of $\mbE$.\\
On the other hand, by (ET4)$^{\op}$, we have a commutative diagram 
\[\begin{tikzcd}
	D & D \\
	M & E & J & {} \\
	C & G & J & {} \\
	{} & {}
	\arrow[equals, from=1-1, to=1-2]
	\arrow["u"', from=1-1, to=2-1]
	\arrow["d", from=1-2, to=2-2]
	\arrow["v", from=2-1, to=2-2]
	\arrow["x"', from=2-1, to=3-1]
	\arrow["y", from=2-2, to=2-3]
	\arrow["e", from=2-2, to=3-2]
	\arrow["\theta", dashed, from=2-3, to=2-4]
	\arrow[equals, from=2-3, to=3-3]
	\arrow["c"', from=3-1, to=3-2]
	\arrow["{\varepsilon_2\cdot c=i\cdot\varepsilon_1}"', dashed, from=3-1, to=4-1]
	\arrow["f"', from=3-2, to=3-3]
	\arrow["{\varepsilon_2}", dashed, from=3-2, to=4-2]
	\arrow["{\eta_3}"', dashed, from=3-3, to=3-4]
\end{tikzcd}.\] 
Observe that the commutativity $(\ast)$ in diagram \eqref{eq:3x3(*)} yields $y=fe=hl$. Given that $\mbF$ is closed and $h$,$l$ are both $\mfs|_{\mbF}$-deflations, it follows by Proposition \ref{HLN21:prop 3.16} that $y=hl$ is also an $\mfs|_{\mbF}$-deflation and thus $v$ is an $\mfs|_{\mbF}$-inflation from Lemma \ref{lemma: Finfl-Fdefl}. Now we get that $d=vu$ is a composition of $\mfs|_{\mbF}$-inflations which again by Proposition \ref{HLN21:prop 3.16} implies that $d$ is an $\mfs|_{\mbF}$-inflation.
\end{proof}

As a consequence of the previous result, we get the following analogous of \cite[Corollary 3.6]{Bu} given in the context of exact categories.

\begin{corollary} Let $(\mcC,\mbE,\mfs)$ be an extriangulated category and let $\mbF$ be a closed subfunctor of $\mbE$. Consider a commutative diagram of $\mbE$-triangles 
\begin{equation}\label{3x3 diagram}
\begin{tikzcd}
	A & B & C & {} \\
	D & E & G & {} \\
	H & I & J & {} \\
	{} & {} & {}
	\arrow["a", from=1-1, to=1-2]
	\arrow["i"', from=1-1, to=2-1]
	\arrow["b", from=1-2, to=1-3]
	\arrow["j", from=1-2, to=2-2]
	\arrow["{\varepsilon_1}", dashed, from=1-3, to=1-4]
	\arrow["c", from=1-3, to=2-3]
	\arrow["d"', from=2-1, to=2-2]
	\arrow["k"', from=2-1, to=3-1]
	\arrow["e"', from=2-2, to=2-3]
	\arrow["l", from=2-2, to=3-2]
	\arrow["{\varepsilon_2}", dashed, from=2-3, to=2-4]
	\arrow["f", from=2-3, to=3-3]
	\arrow["g"', from=3-1, to=3-2]
	\arrow["{\eta_1}"', dashed, from=3-1, to=4-1]
	\arrow["h"', from=3-2, to=3-3]
	\arrow["{\eta_2}", dashed, from=3-2, to=4-2]
	\arrow["{\varepsilon_3}", dashed, from=3-3, to=3-4]
	\arrow["{\eta_3}", dashed, from=3-3, to=4-3]
\end{tikzcd}\end{equation}
where all columns are $\mbF$-triangles, such that $(i,c):\varepsilon_1 \to \varepsilon_2$ and $(k,f):\varepsilon_2 \to \varepsilon_3$ are morphisms of $\mbE$-extensions. Then, the following conditions are equivalent.
\begin{enumerate}
    \item The middle row of (\ref{3x3 diagram}) is an $\mbF$-triangle.
    \item The first and third rows of (\ref{3x3 diagram}) are $\mbF$-triangles.
\end{enumerate}

\end{corollary}
\begin{proof} Since $\mbF$ is a closed subfunctor of $\mbE$, it follows from Theorem \ref{thm:3x3 equiv closed} that $\mbF$ has the $3\times 3$-lemma property. Hence, (2) $\Rightarrow$ (1) holds.

Let us suppose that $\varepsilon_2$ is an $\mbF$-extension. Then, we have that $\varepsilon_3 \cdot f = k \cdot \varepsilon_2 \in \mbF(G,H)$ and $i\cdot \varepsilon_1 = \varepsilon_2 \cdot c \in \mbF(C,D)$ with $f$ an $\mfs|_\mbF$-deflation and $i$ an $\mfs|_\mbF$-inflation. Therefore, from \cite[Lemma 3.14]{HLN21} and its dual we can conclude that $\varepsilon_1$ and $\varepsilon_2$ are $\mbF$-extensions.
\end{proof}

\begin{remark} After a careful revision of the proof of Theorem \ref{thm:3x3 equiv closed}, we observe that it remains true when we assume that $(k,l,f)$ is a morphism of $\mbE$-triangles instead of $(i,j,c)$ being one. The proof follows by using dual arguments.
\end{remark}

\noindent \textbf{Relation with saturated proper classes.} We finish this work giving outcomes which relates the developed theory with saturated and proper classes given in \cite{beligiannis2000relative,HZZ20}. We begin by recalling the following definition.

\begin{definition}\cite[Definition 3.1]{HZZ20}. Let $(\mcC,\mbE,\mfs)$ be an extriangulated category and $\xi$ a class of $\mbE$-triangles closed under isomorphisms. We say that $\xi$ is a \textbf{proper class} if the following conditions hold.
\begin{enumerate} 
    \item $\xi$ contains all the split $\mbE$-triangles.
    \item $\xi$ is closed under finite coproducts.
    \item $\xi$ is closed under base change and co-base change.
\end{enumerate}
Moreover, $\xi$ is \textbf{saturated} if it satisfies the following condition: in the situation of the diagram (\ref{diag:NP3.15}), if the third column and the middle row are $\mbE$-triangles in $\xi$, then the third row also belongs to $\xi$.
\end{definition}

Some references, as in \cite{HZZ20}, consider the saturatedness condition as part of the definition of proper classes. In this work, a proper class will not necessarily be saturated (see \cite[Appendix A]{sakai2023relative}). 

Let $(\mcC,\mbE,\mfs)$ be an extriangulated category and let $\xi$ be a class of $\mbE$-triangles closed under isomorphisms. For any $A,C \in \mcC$, set 
\[ \mbE_{\xi}(C,A):=\{\delta \in \mbE(C,A) \mid \text{ there exists an }\mbE\text{-triangle } A \to B \to C \xdashrightarrow{\delta} \in \xi \; \}.  \]
It was proven in \cite[Theorem 3.2]{HZZ20} that $\xi$ is a saturated proper class if, and only if, $(\mcC,\mbE_{\xi},\mfs|_{\mbE_{\xi}})$ is an extriangulated category, and then $\mbE_{\xi}$ is a closed subfunctor of $\mbE$ from Proposition \ref{HLN21:prop 3.16}. Moreover, this correspondence is bijective as shown in the following proposition.

\begin{proposition}\cite[Proposition A4.]{sakai2023relative} \label{prop:csubf-propercls}
Let $(\mcC,\mbE,\mfs)$ be an extriangulated category. Then, the following conditions hold.
\begin{enumerate}
    \item The correspondence $\xi \mapsto \mbE_{\xi}$ between proper classes of $\mbE$-triangles and additive subfunctors of $\mbE$ is bijective.
    \item In the above correspondence,  saturated proper classes corresponds to closed subfunctors of $\mbE$.
\end{enumerate}
\end{proposition}

The previous result shows that saturated proper classes and closed subfunctors are essentially the same. Thus, we may establish our main result in terms of proper classes as follows.

\begin{corollary} 
Let $(\mcC,\mbE,\mfs)$ be an extriangulated category. If $\xi$ is a proper class of $\mbE$-triangles, then following conditions are equivalent.
\begin{enumerate}
    \item  $\xi$ is saturated.
    \item  $\xi$ satisfy the \textbf{$3 \times 3$-lemma property}: for any commutative diagram of $\mbE$-triangles
    \begin{equation} \label{3x3 diagram pcls}
    \begin{tikzcd}
	A & B & C & {} \\
	D & E & G & {} \\
	H & I & J & {} \\
	{} & {} & {}
	\arrow["a", from=1-1, to=1-2]
	\arrow["i"', from=1-1, to=2-1]
	\arrow["b", from=1-2, to=1-3]
	\arrow["j", from=1-2, to=2-2]
	\arrow[dashed, from=1-3, to=1-4]
	\arrow["c", from=1-3, to=2-3]
	\arrow["d"', from=2-1, to=2-2]
	\arrow["k"', from=2-1, to=3-1]
	\arrow["e"', from=2-2, to=2-3]
	\arrow["l", from=2-2, to=3-2]
	\arrow[dashed, from=2-3, to=2-4]
	\arrow["f", from=2-3, to=3-3]
	\arrow["g"', from=3-1, to=3-2]
	\arrow[ dashed, from=3-1, to=4-1]
	\arrow["h"', from=3-2, to=3-3]
	\arrow[ dashed, from=3-2, to=4-2]
	\arrow[ dashed, from=3-3, to=3-4]
	\arrow[ dashed, from=3-3, to=4-3]
\end{tikzcd}     
\end{equation}
  such that $(i,j,c)$ or $(k,l,f)$ is a morphism of $\mbE$-triangles, if all columns and the outer rows belong to $\xi$, then also the middle row belongs to $\xi$.
\end{enumerate}
\end{corollary}

\begin{corollary} Let $(\mcC,\mbE,\mfs)$ be an extriangulated category and let $\xi$ be a saturated proper class of $\mbE$-triangles. Consider a commutative diagram of $\mbE$-triangles as in (\ref{3x3 diagram pcls}) such that all columns belong to $\xi$, and the triplets $(i,j,c)$ and $(k,l,f)$ are morphisms of $\mbE$-triangles. Then, the following conditions are equivalent.
\begin{enumerate}
    \item The middle row of (\ref{3x3 diagram pcls}) belongs to $\xi$.
    \item The first and third rows of (\ref{3x3 diagram pcls}) belong to $\xi$.
\end{enumerate}
\end{corollary}



\section*{\textbf{Acknowledgments}} 

We thank Mindy Huerta and Octavio Mendoza for their encouragement, feedback and helpful improvements on the first version of this work.



\bibliographystyle{abbrv} 
\bibliography{biblio23} 





\end{document}